\documentclass[10pt,a4paper]{amsart}
\usepackage{hyperref}
\usepackage{mathrsfs}
\usepackage{amsmath}
\usepackage{amssymb}
\usepackage[all]{xy}
\usepackage{latexsym}
\usepackage{ bbold }

\mathchardef\dashmod="2D

\usepackage{datetime}
\usepackage{pgf,tikz}
\usetikzlibrary{arrows}

\usepackage[latin1]{inputenc}
\usepackage[T1]{fontenc}

\newcommand{ \inj}{ \hookrightarrow}
\newcommand{ \surj}{ \twoheadrightarrow}

\newcommand{\coker}{\mathrm{coker}}

\newcommand{\ext}{\mathrm{Ext}}

\renewcommand{\hom}{\mathrm{Hom}}
\newcommand{\sthom}{\underline{\mathrm{Hom}}}
\newcommand{\stext}{\underline{\mathrm{Ext}}}

\newcommand{\Z}{\mathbb{Z}}
\newcommand{\F}{\mathbb{F}}

\newcommand{\E}{\mathcal{E}}

\newcommand{\st}{\mathcal{S}\mathrm{tab}}
\newcommand{\stinf}{\mathcal{S}\mathrm{tab}_{\infty}}

\newcommand{\gm}{\mathbb{G}_m}
\newcommand{\Pic}{\mathrm{Pic}}

\newcommand{\eq}{\cong}
\newcommand{\steq}{\simeq}

\newcommand{\colim}{\mathrm{colim}}

\newcommand{\holim}{\mathrm{holim}}

\newcommand{\A}{\mathcal{A}}
\newcommand{\B}{\mathcal{B}}
\renewcommand{\P}{\mathcal{P}}
\newcommand{\C}{\mathcal{C}}
\newcommand{\D}{\mathcal{D}}

\newcommand{\I}{\mathcal{I}}
\newcommand{\un}{\mathbb{1}}

\newcommand{\Amod}{{ }_{A}\textbf{Mod}}

\newcommand{\spc}{\mathrm{Spc}}
\newcommand{\supp}{\mathrm{Supp}}
\newcommand{\proj}{\mathbf{Proj}}

\newcommand{\ho}{\mathit{ho}}

\theoremstyle{definition}
\newtheorem{de}{Definition}[section]
\newtheorem{nota}[de]{Notation}

\theoremstyle{plain}
\newtheorem{thm}[de]{Theorem}
\newtheorem{lemma}[de]{Lemma}
\newtheorem{pro}[de]{Proposition}
\newtheorem{cor}[de]{Corollary}

\newtheorem{warning}[de]{Warning}

\newtheorem*{thm*}{Theorem}
\newtheorem*{lemma*}{Lemma}
\newtheorem*{pro*}{Proposition}
\newtheorem*{cor*}{Corollary}

\theoremstyle{remark}
\newtheorem{rk}[de]{Remark}
\newtheorem{ex}[de]{Example}

\title[Picard Hopf]{Local study of stable module categories via tensor triangulated geometry}
\author{Nicolas Ricka}
\address{Department of Mathematics, Wayne State University,
 Detroit, MI 48202}
\email{nicolas.ricka@wayne.edu}
\thanks{The author thanks Antoine Touze for the numerous discussions that were the starting point of this project, and Bob Bruner for his remarks and comments about stable module categories throughout the writing of this paper}

\begin{document}

\begin{abstract}
We investigate the particular properties of the stable category of modules over a finite dimensional cocommutative graded connected Hopf algebra $A$, via tensor-triangulated geometry. This study requires some mild conditions on the Hopf algebra $A$ under consideration (satisfied for example by all finite sub-Hopf-algebras of the modulo $2$ Steenrod algebra). In particular, we study some particular covers of its spectrum of prime ideals $\spc(A)$, which are related to Margolis' Work.

We then exploit the existence of Margolis' Postnikov towers in this situation to show that the localization at an open subset $U$ of $\spc(A)$, for various $U$, assembles in an $\infty$-stack. 

Finally, we turn to applications in the study of Picard groups of Hopf algebras and localizations in the stable categories of modules.
\end{abstract}

\maketitle

\section*{Introduction}

Let $A$ be a finite dimensional sub-Hopf-algebra of the Steenrod algebra. From the category of graded modules over it, one can construct the so-called stable module category, which is obtained by killing the projective modules.

The various properties of the stable category of modules over sub-Hopf-algebras of the Steenrod algebra have striking consequences in stable homotopy theory, since this category governs the $E_2$-page of the Adams spectral sequence. For instance, localization and periodicities can be seen through the eyes of the stable categories of modules (see \cite{Pal99}).
Some invariant of this purely algebraic category can moreover have surprising consequences, for instance, when $A$ is the sub-Hopf algebra of the modulo $2$ Steenrod algebra generated by $Sq^1$ and $Sq^2$, the computation of the Picard group (the group of invertible elements with respect to the tensor product) in the stable category of $\A(1)$-modules is the key element in the proof of the unicity of $bo$ given by Adams and Priddy in \cite{AP76}.

The stable category of $A$-modules is highly structured. Indeed, it is the homotopy category of a stable monoidal infinity category. Denote it $\stinf(A)$.

In particular, its associated homotopy category $st(A):= \ho(\stinf(A))$ is a tensor-triangulated category, and it makes sense to talk about its spectrum $\spc(A)$ (in the sense of Balmer, see \cite{Bal04}). Moreover, for any  open subset $U$ of $\spc(A)$, we can consider the localization functor:
$$ \st(A) \rightarrow \st(A)(U).$$

 Our aim is to play with both the tensor-triangulated structure on the homotopy category $\st(A)$ and the infinity-category structure on $\stinf(A)$. 
 
 For generalities about tensor-triangulated geometry, we refer to the work of Balmer \cite{Bal04,Bal10}. In this paper, we are settling gluing results for such categories. Those types of results are already studied in great generality in the work of Balmer and Favi \cite{Bal07,BF11} (we will recall their main results in Section \ref{sec:ttg}). Note however that since we restrict our attention to stable categories of modules, we are able to obtain a much more rigid gluing result in this case.
 
The main result here is that there is a good Grothendieck topology (called the segmental topology, because its opens are associated to sequences of contiguous Margolis operations, see Subsection \ref{subsec:segmentalopen}), such that

\begin{thm*}[Theorem \ref{thm:inftystack}]
The functor
$$Open(\spc(A))^{op} \rightarrow \text{stable}\dashmod\infty\dashmod\otimes\dashmod Cat,$$
which sends an open $U$ to $\stinf(A)(U)$ is a stack (of $\infty$-categories) for the segmental topology. 
\end{thm*}

In clear, this means that the local categories of modules glue well along any segmental cover (that is, a cover for the segmental topology). The reader who is familiar with the stable category of $\E(1)$-modules, or $\A(1)$-modules recognizes in this result the classical approach to study $\st(A)$ in these cases: considering first the categories of $Q_0$-acyclic (resp. $Q_1$-acyclic) modules, and then adressing the extension problems (see \cite{Ad74,Ad71,AM74,Mar83}).

To show this, we essentially use the gluing results of Balmer-Favi \cite{Bal07}. This result holds in great generality, but fail to produce actual gluings of $\infty$-categories without extra assumptions. Here, we will profit from some additional structure which exist in our specific context: Margolis' Postnikov towers for modules, as constructed first by Margolis \cite{Mar83}, and used heavily in \cite{Pal99}.

A first application of this is the existence of Mayer-Vietoris spectral sequences for the computation of Picard groups (Theorem \ref{thm:mayervietoris}). As the hypothesis of this Theorem depends on the actual definition of the segmental topology, which is far too involved for this introduction, we will only write down a consequence of this Theorem :

\begin{pro*}[Proposition \ref{pro:localdetection}]
For all partition $1 = i_0 < i_1 < \hdots i_n = N$ of the Margolis operations in $A$, the Picard groups of the stable category of $A$-modules is detected over $U_{[i_k,i_{k+1}]}$, meaning that the morphism of abelian groups
\begin{equation*}
\Pic(A) \rightarrow \bigoplus_k \Pic(U_{[i_k,i_{k+1}]}),
\end{equation*}
where $\Pic(U_{[i_k,i_{k+1}]})$ is the Picard group in the local category $\st(A)(U_{[i_k,i_{k+1}]})$, is injective.
\end{pro*}

The previous proposition generalizes Bob Bruner's results in the category of $\A(1)$-modules (see \cite{Br14}).

The first part of this paper consists in the recollection about the three theories we use here, namely \begin{itemize}
\item the stable category of modules,
\item tensor-triangulated geometry,
\item infinity-stacks.
\end{itemize}

Note that the point of view adopted here is somewhat different than the one usually adopted in the literature, since \begin{itemize}
\item we emphasize the fact that the stable category actually comes with a canonical stable monoidal infinity category structure, which plays a crucial role in the arguments,
\item the notion of spectrum in tensor-triangulated geometry is the one defined by Balmer, but the associated localizations are understood in the $\infty$-categorical setting,
\item our notion of infinity-stacks is defined to be a sheaf of $\infty$-categories (and not necessarily $\infty$-groupoids, as usually considered).
\end{itemize}

The second part contains the two main results: the identification of the spectrum of the stable category of modules over a large class of Hopf algebras (including the sub-Hopf-algebras of the Steenrod algebra) in Section \ref{sec:spectrumhopf}, and the proof of the main theorem in Section \ref{sec:stacklocal}.

The last part of the paper is dedicated to applications. The reader which is already familiar with the stable category of $\A(1)$-modules can follow Section \ref{sec:toy} in parallel to the rest of the present paper, as this makes explicit all the constructions we are doing in that particular case.

\tableofcontents

\part{Set-up}

\textbf{Conventions : }we always work over the field with two elements, denoted by $\F$. All the Hopf algebras under consideration are cocommutative finite dimensional graded connected Hopf algebras, except when it is explicitely specified otherwise. Capitalized letters $A,B,\hdots$ denotes generic Hopf algebras, and the symbols $\A(n), \E(n)$  denote the classical sub-Hopf-algebras of the (non finite dimensional) Steenrod algebra $\A$ generated by the first $(n+1)$ Steenrod squares $Sq^{2^i}$, for $0 \leq i \leq n$ (respectively the first $(n+1)$ Milnor operations, $Q_i$, for $0 \leq i \leq n$).\\

Let $A$ be a Hopf algebra. The category of finitely generated modules over $A$ is denoted by $\Amod^{nc}$. The decoration $nc$ is here to emphasize the fact that there are some non-compact objects in this category. We choose to emphasize this, as the assumption that all objects are compact is a central assumption in many results of tensor-triangulated geometry.
The full subcategory of $\Amod^{nc}$ whose objects are compact (\textit{i.e.} finitely generated as vector spaces) is simply denoted $\Amod$. In general, a category which has non-compact objects will always be denoted this way, for example $\C^{nc}$, and $\C$ will denote its full subcategory spanned by compact objects. \\

Finally, we use the term $\otimes$-triangulated category as a shorthand for a symmetric monoidal closed triangulated category, in which the tensor is an exact functor in both variables. In such categories, we will stick to the notation $\Omega^{-1}$ for the suspension functor, whereas $\Sigma$ denotes the shift (that is, the regrading) in graded categories.

\section{The stable module category}

\subsection{The stable module category of a finite dimensional Hopf algebra as a stable $\infty$-$\otimes$-category}

Let $A$ be a finite dimensional cocommutative graded unital Hopf algebra. We first recall briefly the definition of the stable category of $A$-modules.

\begin{de}
We say that a map $f: X \rightarrow Y$ in the category $\Amod^{nc}$ is \begin{itemize}
\item a weak equivalence if both its kernel and cokernel are free modules,
\item a cofibration if $f$ is injective,
\item a fibration if $f$ is surjective.
\end{itemize}
\end{de}

\begin{rk}
Since a finite dimensional Hopf algebra is in particular a Frobenius algebra, the notions of free, projective and injective coincide.
\end{rk}

\begin{pro}
The classes of fibration, cofibration, and weak equivalences give to $\Amod^{nc}$ the structure of a simplicial monoidal model category.

Moreover, its homotopy category is the usual stable model category of $A$-modules.
\end{pro}

\begin{proof}
This is \cite[Example 2.4.(v)]{SS03}.
\end{proof}

\begin{cor} \label{cor:hofiber}
The homotopy fiber of a surjective map is its kernel, and the homotopy cofiber of an injective map is its cokernel.
\end{cor}

\begin{de}
Let $\st(A)^{nc}$ be the homotopy category of $\Amod^{nc}$.
\end{de}

As $\Amod^{nc}$ is a simplicial model category, there is an $\infty$-category associated to it, which captures the homotopy theory described by $\Amod^{nc}$ (see for example \cite[Proposition A.3.7.6]{HTT}, although we need a finer result here, to encompass the stability and tensor structure (see \cite[Proposition 1.5.6]{HAII})). \\

The tensor product over $\F$ of $A$-modules can be given an action of $A$, induced by the coproduct $\Delta : A \rightarrow A \otimes A$. This tensor product of $A$-modules is part of a symmetric monoidal closed structure on $\Amod$.

\begin{de}
Let $\un$ be the unit for this monoidal structure, that is, the one dimensional $A$-module concentrated in degree zero.
The functors $\otimes$ and $F(-,-)$ denotes the tensor product (over $\F$) and the function $A$-module respectively.
\end{de}

Recall that the category $\st(A)^{nc}$ is a stable category, in the sense of \cite[Section 9.6]{HPS}. In particular, it is a triangulated category. The suspension functor has the following explicit description.

\begin{de}
Let $X$ be an $A$-module. The shift (or suspension) in the category $\st(A)^{nc}$ is 
$$ \Omega^{-1} X := \coker( X \rightarrow I)$$
where $X \rightarrow I$ is the injective envelope of $X$.
\end{de}

\begin{pro}
The functor $\Omega^{-1}$ is invertible in the stable category. Moreover, for all $n \in \Z$, there is a canonical weak equivalence
$$ \Omega^nX \simeq \Omega^n \un \otimes X.$$
\end{pro}

\begin{pro}
The monoidal structure on $\Amod^{nc}$ gives it the structure of a simplicial monoidal model category. Let $\stinf(A)^{nc}$ the associated stable $\infty$-$\otimes$-category.

In particular, the homotopy category of $ \stinf(A)^{nc}$ is canonically equivalent to $\st(A)^{nc}$, as a symmetric monoidal triangulated category.
\end{pro}

\begin{proof}
The construction of the $\infty$-$\otimes$-category associated to a symmetric monoidal model category such as $\Amod^{nc}$ is \cite[Proposition 1.5.6]{HAII}. The hypothesis of this proposition is checked using the explicit description of the model structure we have here.
\end{proof}

\subsection{Interpretation and structure of the homotopy groups of morphism spectra}

A pleasant property of the $\infty$-category $\stinf(A)^{nc}$ is that the homotopy groups of its spaces of morphisms have an algebraic description.

\begin{de}
Let $X$ and $Y$ be two $A$-modules. We make the following conventions: \begin{itemize}
\item  $\sthom_A(X,Y)$ denotes the spectrum of morphisms in $\stinf(A)^{nc}$,
\item $\stext^i_A(X,Y)$ is the group of extensions of $A$-modules $\ext^{1}_A(\Omega^{i-1}X,Y)$.
\end{itemize}
\end{de}

\begin{rk}
By construction, the functor $\stext^i_A(X,Y)$ coincides with the usual extension groups whenever $i >0$. Whenever $i=0$, this functors coincides also with $\hom_{\st(A)}(X,Y)$.

From another perspective, the functor $\stext^i_A(X,Y)$ is also the Tate functor associated to $\ext^i_A(X,Y)$.
\end{rk}

\begin{pro} \label{pro:pihomandext}
There is a natural isomorphism
$$ \pi_i(\sthom_A(-,-)) \cong \stext^{-i}_A(-,-)$$
\end{pro}

\begin{proof}
The loop module  $\Omega X$ enters a homotopy pullback square in $\stinf(A)^{nc}$
$$\xymatrix{ \Omega X \ar[r] \ar[d] & \ar[d] 0 \\
0 \ar[r] & X } $$
identifying $\sthom_A(\Omega X,Y)$ with the suspension of $\sthom_A(X,Y)$. The identification of the homotopy group  $\pi_0(\sthom_A(X,Y))$ being clear, the result follows.
\end{proof}

We now make a few classical observations about this stable category. Recall that $H^{*,*}(A)$ denotes the bigraded abelian group $\ext^{*,*}_A(\un,\un)$, where the first grading is the homological grading, and the second one is the internal grading.

\begin{lemma}
The commutative algebra $H^{*,*}(A)$ is finitely generated. Moreover, for all $X$, $Y$, $\ext_A^{*,*}(X,Y)$ is an $H^{*,*}(A)$-module.
\end{lemma}

\begin{proof}
The first assertion is \cite[Theorem A]{Wil81}. For the second part, the action is given by the monoidal structure.
\end{proof}

\subsection{Margolis Hopf algebra}

The role of particular sub-Hopf algebras of $A$, its so-called quasi-elementary sub-Hopf algebras, is central in the study of the stable category of modules over it. However, we will not stress out this point here (neither give the actual definition of a quasi-elementary sub-Hopf algebra), since we will quickly reduce to the case where these sub-algebra are simply exterior (Definition \ref{de:margolisalg}), and stick to this case. \\

Let us recall briefly why the quasi-elementary sub-Hopf algebras of $A$ captures the various features of the stable category of $A$-modules.

Note that for any sub-Hopf algebra $B$ of $A$, there is an induced forgetful functor
\begin{equation}
U_B : \stinf(A) \rightarrow \stinf(B).
\end{equation}
The general slogan is that all the features of the stable category of modules over $A$ are detected after restricting to the collection $\stinf(E)$, where $E \subset A$ run through the quasi-elementary sub-Hopf algebras of $A$. The reader interested in an $\infty$-categorical formulation of this heuristic can take a look at \cite{Ric16}.

\begin{pro}[Theorem 1.2-1.4, \cite{Pal97}]
The quasi-elementary sub-Hopf algebras of $A$ detects the following properties: \begin{itemize}
\item Nilpotence in $\ext(\un, \un)$, and in $\ext(X,X)$, for any module $X$,
\item Triviality of an $A$-module.
\end{itemize}
Moreover, there is an isomorphism modulo nilpotents
$$  H^{*,*}(A) \eq \lim_B H^{*,*}(B),$$
where the limit is taken over all the quasi-elementary sub-Hopf-algebras of $A$.
\end{pro}

However, finding all the quasi-elementary sub-Hopf algebras of a given Hopf algebra $A$ is a difficult problem in general. For this reason, we will restrict our attention to a particular case.

\begin{de}
A sub-Hopf-algebra $E$ of $A$ is called elementary if it is commutative, and $x^2=0$ for all $x$ in the augmentation ideal of $E$.
Denote $\E$ (or $\E_A$, if there is an ambiguity) the category of elementary sub-Hopf algebras of $A$.
\end{de}

\begin{lemma}
Elementary sub-Hopf-algebras of $A$ are quasi-elementary.
\end{lemma}

\begin{proof}
It is stated as a lemma, but is really a consequence of the definition \cite[Definition 1.1]{Pal97}.
\end{proof}

\begin{de} \label{de:margolisalg}
A Hopf algebra is Margolis if \begin{enumerate}
\item all its quasi-elementary sub-Hopf-algebras are in fact elementary (that is exterior),
\item one can arrange the generators of these elementary subalgebras in increasing degree $$p_1, \hdots p_N.$$
\end{enumerate}
\end{de}

This sounds like a restrictive definition. However, this covers a large spectrum of natural examples. For instance, in the ungraded case, all group algebras satisfied the property expressed in the first point of Definition \ref{de:margolisalg} by Quillen's stratification result.
This property is also satisfied by the Steenrod algebra  (it is an infinite dimensional Margolis Hopf algebra), and all of its sub-Hopf algebra are Margolis as well.

It is possible to be completely explicit in the determination of all the elementary sub-Hopf algebras of the sub-algebras of the Steenrod algebra using profile functions (see \cite{Pal99}). We will see how this works in the case of $\A(1)$ in Section \ref{sec:toy}.

\subsection{Margolis homology}

We first recall the classical definition of Margolis homology in the category of $A$-modules.

\begin{de}
Let $A$ be a Margolis Hopf algebra. Let $p_k$ be one of its Margolis operations. Note that by definition,  $p_k^2=0$. The $p_k$-Margolis homology functor, denoted by $ H^*(-,p_k),$ is the functor
$$ \frac{ker(p_k)}{im(p_k)}.$$
\end{de}

\begin{lemma}
For all $i \in \Z$, the functors $H^*(-,p_k)$ and $\stext^{i,*-i|p_k|}_{\Lambda(p_k)}(\un,X)$ are naturally isomorphic.
\end{lemma}

\begin{proof}
This can be easily seen using the usual periodic resolution of $\un$ in the category of $\Lambda(p_k)$-modules.
\end{proof}

\begin{cor}
The $p_k$-Margolis homology sends a short exact sequences of $A$-modules $$X \inj Y \surj Z$$ to a long exact sequence
$$ \hdots \rightarrow H^{*-|p_k|}(N,p_k) \rightarrow H^*(X,p_k) \rightarrow H^*(Y,p_k) \rightarrow H^*(Z,p_k) \rightarrow \hdots.$$
\end{cor}

\begin{pro} \label{lemma:margolisdetect}
Let $X$ be a module over a Margolis Hopf algebra $A$. Suppose moreover that $X$ is either bounded below or bounded above. Then $X$ is free if and only if, for all $1 \leq k \leq N$, 
$$H(X;p_k)=0.$$
\end{pro}

\begin{proof}
We already know by \cite[Theorem 1.3]{Pal97} that $X$ is free if and only if $X$ is free over all the quasi-elementary sub-Hopf algebras of $A$. The fact that this is in turn equivalent to the triviality of all the Margolis cohomology groups is a classical result, see for example \cite[Theorem 2.1]{Ad71}.
\end{proof}

\begin{lemma} \label{lemma:kunneth}
Let $k \in [1,N]$. Margolis homology with respect to $p_k$ satisfies a K\"unneth formula, that is for all $X,Y \in \st(A)^{nc}$, there is an isomorphism
$$ H^*(X, p_k) \otimes H^*(Y, p_k) \cong  H^*(X \otimes Y, p_k).$$
\end{lemma}

\begin{proof}
The proof is the same as in the classical case for the K\"unneth formula for chain complexes.
\end{proof}

\begin{de} \label{de:local}
We say that the Margolis homology of an $A$-module $M$ is concentrated on $S \subset [1,N]$ if 
$$H^*(M,p_i) = 0$$
whenever $i \not\in S$.

A module is $p_i$-acyclic if its Margolis homology is concentrated over $[1,N] - \{ p_i\}$.
\end{de}

\section{Tensor triangulated geometry} \label{sec:ttg}

\subsection{Recollections on tensor triangulated geometry}

Let $\C$ be an essentially small stable  $\infty$-$\otimes$-triangulated category. We suppose moreover that every object of $\C$ is strongly dualizable. \\
As we have seen before, an example of such a category is $\stinf(A)$, the stable module category of a finite dimensional connected Hopf algebra $A$.

We will recall here the basic definitions of tensor triangulated geometry. The purpose of this subsection is to fix the notations. The definitions and results can be found in \cite{Bal04,Bal10}.

Although we will take advantage of the $\infty$-category structure of $\C$ later on, the definition of the spectrum of prime ideals in $\C$ we use here coincides with the spectrum of prime ideals of $ho(\C)$ (with the same definition as the original one introduced by Balmer). In particular, the spectrum itself depends only on the homotopy category of $\C$ as a tensor triangulated category.

The $\infty$-categorical enhancement we can afford here is a version of the localization at $U$, for $U$ an open of $\spc(\C)$. For tensor triangular categories this was done by a Verdier quotient, whereas in our context, we can use the cofiber in $\infty$-categories. The reader who is already familiar with tensor triangulated geometry can safely skip to subsection \ref{sub:infiniverdier}.

\begin{de}
Let $\B \subset \C$ be a full subcategory of $\C$ containing $0$.
\begin{itemize}
\item The category $\B \subset \C$ is thick if it is stable by extensions and retracts.
\item The category $\B$ is an ideal of $\C$ if for all $b \in \B$ and $c \in C$, $b \otimes c \in \B$.
\item A proper ideal $\B$ of $\C$ is a prime ideal if $c \otimes c' \in \B$ implies that either $c$ or $c'$ belongs to $\B$.
\end{itemize}
\end{de}

\begin{de}
The spectrum of prime ideals of $\C$ is the set $\spc(\C)$ of prime ideals, together with the Zariski topology generated by the open sets
$$ U(c) = \{ \P \in \spc(\C) | a \in \P \},$$
where $c$ is an object of $\C$.
\end{de}

The main property of $\spc(\C)$ is that it is exactly the object needed to construct a support theory of $c \in \C$: it is the universal support datum (see \cite{Bal04}).

\begin{de}
Let $c \in \C$.
The closed set $\supp(c) = U(c)^c$ is called the support of the object $c \in \C$. An object $c \in \C$ is supported in $Z$ if $\supp(c) \subset Z$. We denote the full subcategory of $\C$ whose objects are supported on $Z$ by $\C_{Z}$.
\end{de}

\begin{lemma}[\emph{\cite[Proposition 2.6]{Bal04}}] \label{lemma:support}
The support satisfies the following properties \begin{itemize}
\item $\supp(c) = \emptyset$ if and only if $c \eq 0$,
\item $\supp(\un) = \spc(\C)$,
\item $\supp(c \oplus c') = \supp(c) \cup \supp(c')$,
\item $\supp(c \otimes c') = \supp(c) \cap \supp(c')$,
\item $\supp(b) \subset \supp(a) \cup \supp(c)$ whenever $b$ is an extension of $a$ and $c$.
\end{itemize}
\end{lemma}

\begin{pro}[\emph{\cite[Proposition 11]{Bal10}}]
The space $\spc(\C)$ is quasi-compact, and the $U(c)$ forms an open basis of quasi-compact opens. Moreover, any irreducible open has a unique generic point.
\end{pro}

\subsection{$\infty$-categorical quotients} \label{sub:infiniverdier}

Recall from \cite[Definition 5.4]{BGT} That there is a good $\infty$-categorical model for the Verdier quotients 
\begin{equation*}
\st(A)(U) = \st(A)/\st(A)_{U^c},
\end{equation*}
namely, the cofiber of the functor
$$ \stinf(A)_{U^c} \rightarrow \stinf(A).$$

\begin{pro} \label{pro:inftyverdier}
Let $\stinf(A)(U)$ be the cofiber of the (fully faithful) $\infty$-functor $\stinf(A)_{U^c} \rightarrow \stinf(A)$. This is an $\infty$-category whose associated homotopy category is precisely $\st(A)/\st(A)_{U^c}$. This yields a functor of stable-$\infty\dashmod\otimes$-categories 
$$\stinf(A) \rightarrow \stinf(A)(U).$$
\end{pro}

\begin{proof}
This is a consequence of Definition 5.4, Proposition 5.6, Proposition 5.9 of \cite{BGT}, except for the last assertion. The monoidal behaviour of the quotient and the monoidal symetric structure on $\stinf(A)(U)$ comes from the formula $$\supp(c \otimes c') = \supp(c) \cap \supp(c'),$$ given in Lemma \ref{lemma:support}.
\end{proof}

\begin{de}
Let $U$ be an open subset of $\spc(\C)$. Let $\C_Z$ be the full sub-$\infty$-category of $\C$ whose objects are $Z = U^c$ supported objects.  We define $\C(U) := (\C/\C_Z)^{\sharp}$, the idempotent completion of the quotient $\C/\C_Z$.
This is called the $U$-local category.
\end{de}

The following proposition explains in what sense the local $\infty$-category we described here is the correct $\infty$-categorical version of \cite{Bal04}.

\begin{pro}
Let $U \subset \spc(\C)$ be an open. The $U$-local category $\C(U)$ satisfies the two following properties: \begin{enumerate}
\item its homotopy category $\ho(\C(U))$ is canonically equivalent to the classical one $\ho(\C)(U)$, as defined by Balmer,
\item there is an exact sequence of stable $\infty$-categories
$$ \C(U^c) \rightarrow \C \rightarrow \C(U).$$
\end{enumerate}
\end{pro}

\subsection{Classical results about gluing}

In this subsection, we recall the gluing result obtained by Balmer and Favi in \cite{Bal07}.

\begin{pro}[\emph{\cite[Corollary 5.8, 5.10]{Bal07}}] \label{thm:mvglue}
Let
 $$\xymatrix{ U_1 \cap U_2 \ar[r] \ar[d] & U_1 \ar[d]\\
U_2 \ar[r] & \spc(\C)}$$
be a cover of $\spc(\C)$ by two quasi-compact open subsets.
\begin{itemize}
\item Let $a,b$ be two objects of $\C$, then there is a long exact sequence
$$ \hdots \hom_{\C(U_1 \cap U_2)}(a,\Omega b) \stackrel{\partial}{\rightarrow} \hom_{\C}(a,b) \rightarrow \hom_{\C(U_1)}(a,b) \oplus \hom_{\C(U_2)}(a,b) $$ $$ \rightarrow \hom_{\C(U_1 \cap U_2)}(a,b) \stackrel{\partial}{\rightarrow} \hdots$$
\item Let $a_1 \in \C(U_1)$ and $a_2 \in \C(U_2)$ and an isomorphism $\sigma : a_1 \eq a_2$ in $\hom_{\C(U_1 \cap U_2)}(a_1,a_2)$. Then, there is an object $a \in \C$, unique up to isomorphism, which restricts to $a_i$ in $\C(U_i)$, for $i = 1,2$.
\end{itemize}
\end{pro}

Note that, in the case we are interested in, this is saying something about the spectra of morphisms $\sthom_A(a,b)$, because of the identification provided by Proposition \ref{pro:pihomandext}.

\section{Limits and $\infty$-stacks} \label{sec:stack}

\subsection{Limit of homotopy theories}

We review in this section how to detect that an $\infty$-category is the homotopy limit of a diagram of $\infty$-categories. We will essentially use the results of \cite{B11,B12}.

Recall that a weak equivalence of $\infty$-categories is a Dwyer-Kan equivalence. In particular, a functor
$$F: \C \rightarrow \D$$
is an equivalence of $\infty$-categories if and only if \begin{enumerate}
\item The functor $\ho F$ between the associated homotopy categories is essentially surjective,
\item For each pair of objects $x,y \in \C$, the map
$$ F : \hom_{\C}(x,y) \rightarrow \hom_{\D}(Fx,Fy)$$
is a weak equivalence of simplicial sets.
\end{enumerate}

\begin{pro}[\emph{\cite{B12}}]
A homotopy limit of a diagram of complete Segal spaces is a level-wise homotopy limit.
\end{pro}

In particular, we will use the following criterion to identify homotopy colimits of $\infty$-categories.

\begin{de} \label{de:descent}
Let $\mathcal{U}$ be a small category, and $D : \mathcal{U} \rightarrow \infty\dashmod cat$
a diagram of $\infty$-categories.
A descent datum for an object is: \begin{itemize}
\item an object $d_u \in D(U)$, for all $U \in \mathcal{U}$,
\item a weak equivalence $\alpha_{U,V} : f(d_U) \cong g(d_V)$ whenever there is a diagram
$$ U \rightarrow W \leftarrow V$$
in $\mathcal{U}$
\end{itemize}
such that the appropriate diagram commutes.
\end{de}

\begin{pro}
Let $\mathcal{U}$ be a small category, and $D : \mathcal{U} \rightarrow \infty\dashmod cat$
a diagram of $\infty$-categories. Let moreover $\C$ be an $\infty$-category mapping to the diagram $D$ (\textit{i.e.} there are compatible functors $F_U : \C \rightarrow D(U)$ for all $U \in \mathcal{U}$). Then $\C$ is the homotopy limit of $D$ if and only if the two following conditions are satisfied: \begin{itemize}
\item For all descent datum $(d_U, \alpha_{U,V} )$, there exists an object $c \in \C$ whose image through the functors $F_U$ give rise to  $(d_U, \alpha_{U,V} )$,
\item for all pair of objects $x,y \in \C$, the map
$$ \hom_{\C}(x,y) \rightarrow \holim_U  \hom_{\D(U)}(x,y)$$
is a weak equivalence of simplicial sets.
\end{itemize}
\end{pro}

\subsection{$\infty$-stacks, a.k.a. $(\infty,1)$-sheaves}

We now turn to the definition of $\infty$-stack. For us, it will simply be an appropriate version of a sheaf of $\infty$-categories. 

\begin{de} \label{de:stack}
Let $S$ be a site and 
$$F : S \rightarrow \infty\dashmod cat$$
be a presheaf.
The presheaf $F$ is a stack if for any cover $\{U_{\alpha} \rightarrow U \}$ of $U \in S$, the functor
$$ F(U) \rightarrow \holim_{c \in C(\{U_{\alpha} \rightarrow U \})} F(U_{\alpha})$$
is a homotopy equivalence, where $C(\{U_{\alpha} \rightarrow U \})$ is the Cech nerve associated to the cover $\{U_{\alpha} \rightarrow U \}$.
\end{de}

\begin{rk}
Note that we do not restrict to the case where the $\infty$-sheafs takes values in $\infty$-groupoids, as the definition of $\infty$-stack sometimes assumes.
\end{rk}

The hint that this sort of structure will arise when doing tensor triangulated geometry in stable $\infty\dashmod\otimes$-categories is given by the gluings obtained by Balmer-Favi in \cite{Bal07} (see Proposition \ref{thm:mvglue}).

\part{TTG for Margolis Hopf algebras}

\section{The spectrum of graded Hopf algebras} \label{sec:spectrumhopf}

\subsection{The spectrum of a general Hopf algebra}

We start with a determination of the spectrum of the stable category of graded Hopf algebras. We will freely use $\spc(A)$ as a shorthand for $\spc(\st(A))$.

\begin{de}
Let $R$ be a bigraded commutative unital ring. The space
$$\proj(R)$$
is the set of all bihomogenious prime ideals, except the irrelevant one, together with the usual Zariski topology.
\end{de}

\begin{pro} \label{pro:projadspec}
Let $A$ be a Hopf algebra. There is a homeomorphism
$$ \spc(A) \eq \proj(H^{*,*}(A))$$
which sends a prime ideal $\mathfrak{p} \subset H^{*,*}(A)$ to $\P = \{ M \in \st(A) | \mathfrak{p} \subset ker(H^{*,*}(A) \rightarrow \ext_A^{*,*}(M,M))$.
\end{pro}

\begin{proof}
By \cite[Theorem 6.3]{Bal04}, The spectrum of 
\begin{equation*}
\st(A^{\text{ungraded}}) \eq \proj(H^*(A^{\text{ungraded}})),
\end{equation*}
with an analogous homeomorphism to the graded case. Now, \cite[Corollary 3.7]{HP} gives a poset isomorphism between nonempty tensor-closed thick subcategories of finitely generated modules in $\st(A)$ and nonempty subsets of $Spec(H^*(A))$ (bihomogenious prime ideals) closed under specialization {\it i.e.} closed.
\end{proof}

\subsection{The particular case of Margolis-Hopf algebras}

Let now $A$ be a Margolis-Hopf algebra. We now study the consequences of this hypothesis on the spectrum of the associated stable category, and see how tensor triangulated geometry is related, yet richer, than the study of Margolis operations.

\begin{pro} \label{pro:finitespc}
For a Margolis Hopf algebra $A$, $\spc(A)$ is finite. In particular, any subset is quasi-compact.
\end{pro}

\begin{proof}
By \cite{Pal99}, there is an F-isomorphism between $H^{*,*}(A)$ and $\lim_{E \in \E} H^{*,*}(E)$. For a Margolis Hopf algebra, the latter is simply $\F[p_1, \hdots p_N]$. In particular, the associated projective variety has a finite number of points.
\end{proof}

\begin{cor} \label{cor:mvspca}
Any cover of $\spc(A)$ by two open subsets provides a Mayer-Vietoris gluing, that is, Theorem \ref{thm:mvglue} applies.
\end{cor}

\begin{proof}
Because of Proposition \ref{pro:finitespc}, any cover can be replaced by a finite one.
\end{proof}

Another feature of Margolis-Hopf algebras is that we can actually determine explicitely a family of points in the spectrum, closely related to Margolis' localizations.

\begin{de}
Let $S \subset [1,N]$. Let $\I^S$ be the full-subcategory of $\st(A)$ whose objects have Margolis homology concentrated over $S$, in the sense of Definition \ref{de:local}.
Let $1 \leq k \leq N$. Let $\P_{k}$ be the full-subcategory of $\stinf(A)$ whose objects are $p_k$-acyclic.
\end{de}

\begin{pro}  \label{de:primespi}
The following assertions hold. \begin{itemize}
\item For any $S \subset [1,N]$, $\I^S$ is an ideal of $\stinf(A)$,
\item Suppose that $N >1$. Then for $1 \leq k \leq N$, $\P_{k}$ is a prime ideal.
\end{itemize}
\end{pro}

\begin{proof}
These subcategories are full, contains zero. 
The last properties there are to check follows from the K\"unneth formula provided by Lemma \ref{lemma:kunneth}.
\end{proof}

\begin{rk}
The assumption that $N >1$ is simply to avoid the pathological case when $\P_{k}$ is the whole category (and therefore is not prime).
\end{rk}

\section{The stack of local representations} \label{sec:stacklocal}

Let $A$ be a Margolis-Hopf algebra, and $p_1, \hdots p_N$ its Margolis operations. The reader who wants to have a concrete example in mind through this section can follow Section \ref{sec:toy}, which parallels this one in its study of the specific case of $\A(1)$.

In this section, we show that the functor $$\stinf(A)(-) : Open(\spc(A))^{op} \rightarrow  \infty\dashmod cat,$$
defined in Proposition \ref{pro:inftyverdier}, is an $\infty$- stack (see Section \ref{sec:stack}). \\

There is however a subtlety: as observed in \cite{Bal07}, one cannot hope that the functor $\stinf(A)(-)$ is a stack for the Zariski topology. We will define another topology on on $\spc(A)$, which we call the segmental topology, with respect to which the gluing will be well behaved.

The proof makes heavy use of the Postnikov towers for $A$-modules defined by Margolis. This can be thought of as a particularly well behaved occurence of Balmer and Favi's generalized tensor idempotents (see \cite{BF11}). However, this seems to be very specific to the present situation, since: \begin{enumerate}
\item for a Thomason subset $Y$ of $\spc(A)$, the localization sequence 
$$e(Y) \rightarrow \un \rightarrow f(Y)$$
considered in \cite{BF11} can be arranged to live in the category of bounded below modules, for some specific $Y$, which will serve as building blocks for the segmental topology on $\spc(A)$,
\item the graded dual exchanges the subcategories of bounded below and bounded above modules, and the localization away from $Y$ and localization away from some other Thomason subset $Z$ sequences defined in \cite{BF11}.
\item there is some kind of connectivity hypothesis: there is no map from a $Y$ local object to a $Z$ colocal object.
\end{enumerate}

\begin{rk}
In the case of group algebras (which is orthogonal to the case we are studying here), Balmer constructed such a stack in \cite{Bal15} using the so-called \textit{sipp} topology. However, the \textit{sipp} topology and the one we define here are drastically different. Indeed, the \textit{sipp} topology is a Grothendieck topology on the category of finite $G$-sets, whereas our segmental topology is on $\spc(A)$.
\end{rk}

\subsection{Some interesting subsets of $\spc(A)$} \label{subsec:segmentalopen}

\begin{de}
Let $S \subset [1,N]$.
Let $F_S$ be the subset of $\spc(A)$ defined by 
$$F_{S} := \{ \P | \P\text{ does not contain every module with Margolis homology concentrated on }S\}.$$
\end{de}

\begin{lemma}
The sets $F_{S} \subset \spc(A)$ are Zariski closed.
\end{lemma}

\begin{proof}
We can express $F_{S}$ as the  union 
$$\bigcup_X \supp(X),$$
where the union runs every $A$-module $X$ whose Margolis homology is concentrated over $S$. Since the spectrum of $A$ is finite (see Proposition \ref{pro:finitespc}), this is a finite union of closed subsets.
\end{proof}

\begin{de}
A segmental open of $\spc(A)$ is the complementary of some $F_{[a,b]^c}$, for $1 \leq a \leq b \leq N$. We denote each one of these $U_{[a,b]}$.
\end{de}

\begin{warning} \label{warning:colocalizations}
The importance of such segments in the study of modules over a Hopf algebra is explained by Margolis' killing construction \cite{Mar83,Pal01}. These are functors $\Amod \rightarrow \Amod_{[a,b]}^{nc}$, where $\Amod_{[a,b]}^{nc}$ is the full subcategory of $\Amod^{nc}$ whose Margolis homology is concentrated between $a$ and $b$.

These functors are called localization by Margolis. However, these do not correspond to Bousfield localization functors! These are in general composite of a Bousfield localization and a Bousfield colocalization.
\end{warning}

We now state some useful properties of these sets, which are required to define properly the segmental topology.

\begin{lemma}
Let $S,S'$ be segments of $[1,N]$. The following properties hold. \begin{itemize}
\item $U_{S\cap S'} \subset U_S \cap U_{S'}$.
\item $U_S \cup U_{S'} \subset U_{S\cup S'}$.
\item If $S \subset S'$, then $U_S \subset U_{S'}$.
\end{itemize}
\end{lemma}

\begin{proof}
This uses only the definition of $F_S$.
\end{proof}

The purpose of the definition of the closed sets $F_{[a,b]^c}$ is that the following property holds.

\begin{lemma} \label{lemma:supported}
Let $[a,b]$ be any interval of $[1,N]$, and $X$ be an $A$-module. Then $X$ is supported on $F_{[a,b]^c}$ if and only if the Margolis homology of $X$ is concentrated over $[a,b]^c$.
\end{lemma}

\begin{proof}
Let $a,b,X$ as in the hypothesis. By definition, $X$ is supported on $F_{[a,b]^c}$ if 
$$\supp(X) \subset F_{[a,b]^c}.$$
This is equivalent to the assertion $X$ belongs to any prime ideal which contains all modules whose Margolis cohomology is concentrated over $[a,b]^c$. In particular, $X$ belongs to each $\P_i$ (see Definition \ref{de:primespi}), for $i \in [a,b]$, so that $H^*(X,p_i) = 0$ for those $i$. Thus, the Margolis homology of $X$ is concentrated over $[a,b]^c$.

The converse is trivial.
\end{proof}

\begin{pro} \label{pro:zariskicover}
Let $1 = i_0 \leq i_1 \leq \hdots \leq i_k = N$ a partition of $[1,N]$. Then, 
$$\{ U_{[i_l,i_{l+1}]} \rightarrow \spc(A)\}_{0 \leq l < k}$$
is a cover in the Zariski topology.
\end{pro}

\begin{proof}
Let $M \in \left( \bigcup_{0 \leq l < k} U_{[i_l,i_{l+1}]} \right)^c$. This implies that $M \in F_{[i_l,i_{l+1}]^c}$, for all $0 \leq l < k$. Thus, by Lemma \ref{lemma:supported}, $M$ has no Margolis homology, and $M \steq 0$ by Lemma \ref{lemma:margolisdetect}.
\end{proof}

\subsection{Margolis' model for locally supported objects revisited}

We now turn to the comparison between the local categories $\st(A)(U)$, for a Zariski open $U$, and the categories of bounded below local $A$-modules in the sense of Margolis. Let's first recall the definitions and first properties of the latter. For the proofs in this case, the reader is referred to \cite{Mar83}. One critical hypothesis to apply these results is that the modules under consideration should be bounded below. We start this subsection by some notations to emphasize this point.

\begin{de}
Let $\st(A)^{\geq}$ (resp. $\st(A)^{\leq}$) be the full subcategory of $\st(A)^{nc}$ whose objects are bounded below (resp. bounded above) $A$-modules. The categories $\st(A)^{\leq,nc}_{[a,b]}$ and $\st(A)^{\geq,nc}_{[a,b]}$ are the full subcategories of modules whose Margolis homology is concentrated on $[a,b]$, for $1\leq a \leq b \leq N$.
\end{de}

\begin{pro}[see \cite{Mar83}]
Let $[a,b]$ be a segment of $[1,N]$. \begin{itemize}
\item There are functors 
$$(-)_{[a,b]} : \st(A)^{\leq} \rightarrow \st(A)^{\leq,nc}_{[a,b]}$$
and 
$$(-)_{[a,b]} : \st(A)^{\geq} \rightarrow \st(A)^{\geq,nc}_{[a,b]}$$
such that $M$ and $M_{[a,b]}$ are related by a zig-zag of maps (in either category) which induces an equivalence in Margolis homology $H(-,p_i)$, for $a \leq i \leq b$.
\item If $b = N$, then the functor 
$$(-)_{[a,N]} : \st(A)^{\geq} \rightarrow \st(A)^{\geq,nc}_{[a,N]}$$
is a Bousfield localization, and in particular there is a natural transformation $M \rightarrow M_{[a,N]}$, which induces an equivalence in Margolis homology $H(-,p_i)$, for $a \leq i \leq N$.
\item If $a = 1$, then the functor 
$$(-)_{[1,b]} : \st(A)^{\leq} \rightarrow \st(A)^{\leq,nc}_{[1,b]}$$
is a Bousfield localization, and in particular there is a natural transformation $M \rightarrow M_{[1,b]}$, which induces an equivalence in Margolis homology $H(-,p_i)$, for $1 \leq i \leq b$.
\end{itemize}
\end{pro}

\begin{rk}
Equivalently, localization at $[1,b]$ for bounded above $A$-modules provides a colocalization functor in the category of bounded above $A$-modules. As he needed to work entierly in the category of bounded below $A$-modules, localization at $[a,N]$ and colocalization at $[1,b]$ are the functors Margolis defined.
\end{rk}

\begin{pro} \label{pro:verdierloc}
Let $U_{[a,b]}$ be a segmental open. Then Margolis constructions
$$ (-)_{[a,b]} : \st(A) \rightarrow \st(A)_{[a,b]}^{\leq,nc}$$
and 
$$ (-)_{[a,b]} : \st(A) \rightarrow \st(A)_{[a,b]}^{\geq,nc}$$
factors through $\st(A)(U_{[a,b]})$.
\end{pro}

\begin{proof}
By Lemma \ref{lemma:supported}, the modules $M$ which are supported on $F_{[a,b]^c}$ have Margolis homology concentrated over $[a,b]^c$, so they become zero in $\st(A)_{[a,b]}^{nc}$ (by Lemma \ref{lemma:margolisdetect}). Thus the construction factorizes through the Verdier quotient.
The second statement is analogous.
\end{proof}

\begin{lemma}[see \cite{Mar83}] \label{lemma:compacityloc}
The localization functor $(-)_{[a,N]}$ in compact $A$-modules can be obtained for each module $M$ as a colimit
$$\colim_n sk^{[a,N]}_n(M)$$
where each $sk^{[a,N]}_n(M)$ is compact, and the maps $sk^{[a,N]}_n(M) \rightarrow sk^{[a,N]}_{n+1}(M)$ has a cofiber whose Margolis homology is concentrated over $[a,N]$.
The dual statement holds for the segments $[1,b]$ in the category of bounded above $A$-modules.
\end{lemma}

\begin{proof}
This is a consequence of Margolis' explicit construction of these localization functors (see \cite{Mar83}).
\end{proof}

Even if local bounded below $A$-modules are not compact in general, there is still a natural notion of smallness in this category.

\begin{de}
Let $\st(A)^{\leq}_{[a,b]}$ be the full subcategory of $\st(A)^{\leq,nc}_{[a,b]}$ whose objects are in the  essential image of the functor 
$$(-)_{[a,b]} : \st(A) \rightarrow \st(A)^{\leq,nc}_{[a,b]},$$
obtained as the restriction of Margolis' construction to compact $A$-modules.
We define $\st(A)^{\geq}_{[a,b]}$ accordingly.
\end{de}

\begin{lemma} \label{lemma:locandcompacity}
Let $[a,N]$ be a segment. There is a natural equivalence, for $M,M' \in \st(A)$:
$$ \hom_{\st(A)^{\geq}_{[a,N]}}(M_{[a,N]},M'_{[a,N]}) \eq \colim_n \hom_{\st(A)}(M, sk^{[a,N]}_n(M')).$$
\end{lemma}

\begin{proof}
By adjunction, 
$$ \hom_{\st(A)^{\geq}_{[a,N]}}(M_{[a,N]},M'_{[a,N]}) \eq \hom_{\st(A)^{\geq}}(M,M'_{[a,N]}).$$
 Now, by Lemma \ref{lemma:compacityloc},
$$\hom_{\st(A)^{\geq}}(M,M'_{[a,N]}) \eq \hom_{\st(A)}(M,\colim_n sk^{[a,N]}_n(M')).$$
The result follows by compacity of $M$.
\end{proof}

\begin{pro} \label{pro:comparisonMargolis}
Let $[a,N]$ be a segment. The functor
$$ (-)_{[a,N]} : \st(A)(U_{[a,N]}) \rightarrow \st(A)_{[a,N]}^{\geq}$$
provided by Proposition \ref{pro:verdierloc} is an equivalence of categories.
\end{pro}

\begin{proof}
We actually check that this functor is fully faithful and essentially surjective. \begin{itemize}
\item By definition of $\st(A)_{[a,N]}^{\geq}$, $ (-)_{[a,N]}$ is essentially surjective.
\item Let's show that $ (-)_{[a,N]}$ is full. Pick a map in $$ \hom_{\st(A)^{\geq}_{[a,N]}}(M_{[a,N]},M'_{[a,N]}) \eq \colim_n \hom_{\st(A)}(M, sk^{[a,N]}_n(M')),$$
where the isomorphism comes from Lemma \ref{lemma:locandcompacity}. Recall that for all $n$, the cofiber of $M' \rightarrow sk^{[a,N]}_n(M')$ has Margolis homology concentrated on ${[a,N]}$, that is, by Lemma \ref{lemma:supported}, supported on $F_{[a,N]^c}$. Thus, the map $f$ comes from the cospan
$$ \xymatrix{ & sk^{[a,N]}_n(M') & \\ M \ar[ur]^{\tilde{f}} & & M' \ar[ul] }$$
in $\st(A)(U_{[a,N]})$, for some $n \geq 0$.
\item Let's show that $ (-)_{[a,N]}$ is faithful. Let $f \in \hom_{\st(A)(U_{[a,N]})}(M,M')$ be a map that becomes trivial when we apply $ (-)_{[a,N]}$. Recall that we have a calculus of fractions in the Verdier quotient $\st(A)(U_{[a,N]})$. Moreover, because
$$ \hom_{\st(A)^{\geq}_{[a,N]}}(M_{[a,N]},M'_{[a,N]}) \eq \colim_n \hom_{\st(A)}(M, sk^{[a,N]}_n(M')),$$
by Lemma \ref{lemma:locandcompacity}, the map $f$ can be expressed as a cospan
$$ \xymatrix{ & sk^{[a,N]}_n(M') & \\ M \ar[ur]^{\tilde{f}} & & M' \ar[ul] }$$
in $\st(A)(U_{[a,N]})$, for some number $n$.
In particular, if $f_{[a,N]}$ is zero, then the corresponding map $\tilde{f}$ is the zero map in $\st(A)$, and thus $f=0$.
\end{itemize}
\end{proof}

The dual statement holds, for similar reasons.

\begin{pro}
Let $[1,b]$ be a segment. The functor
$$ (-)_{[1,b]} : \st(A)(U_{[1,b]}) \rightarrow \st(A)_{[1,b]}^{\leq}$$
provided by Proposition \ref{pro:verdierloc} is an equivalence of categories.
\end{pro}

The previous description leads to a concrete description of a piece of datum carried by $U_1 \cap U_2$-local maps, where $U_1$, $U_2$ are segmental opens.

\begin{lemma} \label{lemma:localmaps}
Let $U_1 = U_{[a,b]}$ and $U_2 = U_{[c,d]}$ be two segmental opens. Suppose that  $b<c$. Then, the map
$$ \mathrm{iso}_{\st(A)(U_1 \cap U_2)}(X,Y) \cong \hom_{\st(A)}(\Omega L_{[c,d]}X,L_{[a,b]}Y),$$
which sends $f:X \rightarrow Y$ on the edge of the corresponding gluing to $U_1 \cup U_2$, up to $U_1 \cup U_2$-isomorphism, is well defined.
\end{lemma}

\begin{proof}
Let's be a little more precise about the definition of this map. First, as $X$ and $Y$ are fixed $A$-modules in $\stinf(A)(U_1 \cap U_2)$, we might as well see a map between them as the gluing data between $X$ and $Y$, as seen in $\stinf(A)(U_1)$ and $\stinf(A)(U_2)$ respectively.

By the gluing of two objects(Proposition \ref{thm:mvglue}), we know that there exist a $(U_1 \cup U_2)$-local gluing $Z$. Now, Margolis' Postnikov towers gives us a distinguished triangle
$$  Z'_{[1,b]} \rightarrow Z \rightarrow Z_{[b+1,N]}$$
where $Z'_{[1,b]}$ denotes Margolis localization at $[1,b]$ in the category of bounded below modules. The edge of this cofibre sequence provides a map 
$$\Omega Z_{[b+1,N]} \rightarrow Z'_{[1,b]}.$$

To conclude, write both terms as filtered colimits using Lemma \ref{lemma:compacityloc}. Each term in this colimit is related to the terms apprearing in the colimit for $L_{[c,d]}X$ and $L_{[a,b]}Y$ respectively by $U_1 \cap U_2$-local isomorphisms. This ends the construction.
\end{proof}

\begin{ex}
Consider the situation where $U_1 = U_{[1,b]}$ and $U_2 = U_{[b+1,N]}$. Then, an $U_1 \cap U_2$-local isomorphism $f \in   \mathrm{iso}_{\st(A)(U_1 \cap U_2)}(X,Y)$ can be seen as the gluing datum of \begin{itemize}
\item the $U_1$-local object $X$,
\item the $U_2$-local object $Y$,
\item via the isomorphism $f$.
\end{itemize}
The associated map in $\hom_{\st(A)}(\Omega L_{[c,d]}X,L_{[a,b]}Y)$ is the edge of the corresponding gluing.
\end{ex}

\subsection{The segmental topology}

We know by Balmer-Favi \cite{Bal04} that the Zariski topology has too many covers for the assignment $U \mapsto \C(U)$ to be a stack in general. We now define a new topology, the segmental topology, which is more suitable to the study of Hopf algebras.

\begin{de}
Let $\spc_{seg}(A)$ be the set of prime ideals in $\st(A)$, together with the coverage:
\begin{equation*}
\bigcup_{k} U_{I_k} \rightarrow U_{[a,b]}
\end{equation*}
where $I_k$ is a partition of $[a,b]$. This coverage induces a Grothendieck topology on $\spc(A)$ we call the segmental topology.
\end{de}

\begin{rk}
We already know by Proposition \ref{pro:zariskicover} that these are in particular covers in the Zariski sense.
\end{rk}

\subsection{The stack of local modules}

We now turn to the proof of the main result of this section. This is a manifestation of the idea that the category of $A$-modules can be reconstructed by the local ones, where the local categories span all Margolis' operations. The purpose of Theorem \ref{thm:inftystack} is to make this precise. 

\begin{thm} \label{thm:inftystack}
The functor
$$Open(\spc(A))^{op} \rightarrow \text{stable}\dashmod\infty\dashmod\otimes\dashmod Cat,$$
which sends an open $U$ to $\stinf(A)(U)$ is a stack (of $\infty$-categories) for the segmental topology. 
\end{thm}

\begin{rk}
Even if the philosophy of this result is present in Margolis' book \cite{Mar83}, note that the result gives more information about the stable categories, for two distinct reasons: \begin{itemize}
\item it involves the open
$$U_{[a,b]} \cap U_{[c,d]}$$
for example, which correspond to localization functors which where not previously considered in the litterature (in the foundational book \cite{Mar83}, or the memoir \cite{Pal01} for instance),
\item all the functors 
$$\stinf(A)(U) \rightarrow \stinf(A)(V)$$
appearing are actual localizations  (see Warning \ref{warning:colocalizations}).
\end{itemize}
\end{rk}

The proof of this theorem will occupy the rest of this section. We start by showing Proposition \ref{pro:es} and Proposition \ref{pro:ff}. The proof goes by induction, so fix a segmental cover of the segmental open $U$ by $n$ open sets, and suppose that Proposition \ref{pro:es} and Proposition \ref{pro:ff} holds for any such cover with less than $n$ open.

Recall Definition \ref{de:descent} for descent data, and how   descent data relates to the theorem, by definition of a stack \ref{de:stack}.

Fix a segmental cover $\{U_{\alpha} \rightarrow U\}$ until the end of the section. Note that because the intersection of segmental open are segmental open, we can assume that each $U_{\alpha}$ is a segmental open.

\begin{nota}
Rename the open $U_{\alpha}$ appearing in the fixed cover, in the following way: $U_i := U_{[i_k+1,i_{k+1}]}$, for $i=1 \hdots n$, and $i_0 = 0 < i_1 < \hdots < i_{n+1} = N$.
\end{nota}

We need to show that the canonical functor 
\begin{equation*}
F_{desc} : \stinf(A)(U) \rightarrow \holim_c \stinf(A)(c)
\end{equation*}
is an equivalence of categories.  We will show essential surjectivity (gluing of descent data) and fully faithfulness appart in the two following propositions.

\begin{pro} \label{pro:ff}
The functor $F_{desc}$ is fully-faithful.
\end{pro}

\begin{proof}
We argue by induction on the number of open in the segment cover that $F_{desc}$ induces a $\pi_*$-isomomrphism
$$ \hom(X,X') \rightarrow \hom(F_{desc}X,F_{desc}X').$$

The case of two opens is exactly the gluing for morphisms provided by Corollary \ref{cor:mvspca}, and the five lemma.
To show the induction step, consider the cover $\{ (U_1, \hdots U_{n-1}), U_n \}$ of $U$.
\end{proof}

\begin{pro} \label{pro:es}
Gluing of descent data holds for the functor $F_{desc}$. Moreover, such gluings are unique up to equivalence.
\end{pro}

\begin{proof}
As essential surjectivity is checked in the homotopy categories, we will stay in $\st(A)(U) = ho(\stinf(A)(U))$ in this proof. \\

\textbf{Case $n = 2$:} Let $U_1, U_2$ be a cover of $U$, with associated segments $1 < i_1 < i_2 = N$.
Existence of the gluing is precisely the result of Corollary \ref{cor:mvspca}. However, in our situation, we have more: the existence of canonical Postnikov towers gives that $X$ is equivalent to the homotopy cofiber of the map
\begin{equation*}
f_{1,2} : \Omega L_{[i_1+1,N]} X_2 \rightarrow L_{[1,i_1]} X_1,
\end{equation*}
corresponding to the $(U_1 \cap U_2)$-local isomorphism given in the descent datum.

\textbf{Case $n = 3$:} Let $U_1, U_2, U_3$ be a cover of $U$, with associated segments $1 < n_1 < n_2 < n_3 = N$.
Suppose given $X_1$, $X_2$, and $X_3$, together with
 a $(U_i \cap U_j)$-local isomorphism $\phi_{i,j}$ for all $i < j$, satisfying the cocycle condition $\phi_{1,2} \phi_{2,3} = \phi_{1,3}$ in $U_1 \cap U_2 \cap U_3$.
 
In this situation, there is a lift $X_{12} \in \st(A)$ which is $U_i$-locally isomorphic to $X_i$, for $i=1,2$, and a $(U_1 \cup U_2)\cap U_3$-local isomorphism $\phi_{12,3} : X_{12} \cong X_3$. The gluing $X$ is obtained by the gluing of a cover by two open sets using the descent datum $\phi_{12,3} : X_{12} \cong X_3$.

We still have to prove uniqueness. Suppose that $X, X'$ are two gluings of  $X_1$, $X_2$, $X_3$ and their local isomorphisms.

Consider the stable morphisms $f_{i,j} : \Omega L_{[n_{j-1},n_j]} X_j \rightarrow L_{[n_{i-1},n_i]} X_i$
associated to the local isomorphisms $\phi_{i,j}$. Then $X$ and $X'$ enters the following diagram

\begin{equation*}
\xymatrix{ & L_{[1,n_1]} X_1 \ar[dl] \ar[dr] &  \\
L_{[1,n_2]}X_{12} \ar[r] \ar[d] & L_{[n_{1}+1,n_2]} X_2 \ar@{-->}[u]^{f_{1,2}} & L_{[1,n_2]}X'_{12} \ar[l] \ar[d]  \\
X \ar[r] &  L_{[n_{2}+1,N]} X_3 \ar@{-->}[u]^{f_{2,3}} \ar@{-->}[ul]^{d} \ar@{-->}[ur]_{d'} & X'. \ar[l]
}
\end{equation*}
Moreover, there is an isomorphism $\iota : L_{[1,n_2]}X_{12} \cong L_{[1,n_2]}X'_{12}$ which commutes with the solid arrows displayed in the diagram. There is an isomorphism between $X$ and $X'$ if and only if $d' = \iota d$. Let $g : \Omega L_{[n_{2}+1,N]} X_3 \rightarrow X'_{12}$ be the difference between these two maps. By construction $g$ is trivial over $U_2$, so the composite \begin{equation*}
\Omega L_{[n_{2}+1,N]} X_3 \stackrel{g}{\rightarrow} X'_{12} \rightarrow L_{[n_{1}+1,n_2]} X_2
\end{equation*}
is zero, and factors through $L_{[1,n_1]} X_1$, giving a map $\tilde{g} : \Omega L_{[n_{2}+1,N]} X_3 \rightarrow  L_{[1,n_1]} X_1$. But $X$ and $X'$ are gluing of the same descent datum. In particular, the two morphisms $f_{1,3} : \Omega L_{[n_{2}+1,N]} X_3 \rightarrow  L_{[1,n_1]} X_1$ associated to the $(U_1 \cap U_3)$-local isomorphism $\phi_{1,3} : X_1 \cong X_3$ agree, and thus $\tilde{g} = 0$. Consequently, the gluing $X$ is unique up to isomorphism.

\textbf{General case $n \geq 3$:}

Let $U_1, \hdots, U_n$ be a cover of $U$, with associated segments $1 < i_1 < \hdots < i_n = N$.
We show the general case by induction. Suppose that for every segment cover of $n$ elements, there exist a gluing, which is unique up to isomorphism. 

There is a unique gluing to $U_1 \cup \hdots \cup U_{n-1}$ of the descent datum over $U_1$ up to $U_{n-1}$, say $X_{1, \hdots ,n-1}$. Gluing of morphisms give a (possibly non-unique) $(U_1 \cup \hdots \cup U_{n-1}) \cap U_n$-local isomorphism
\begin{equation*}
X_{1, \hdots ,n-1} \cong X_n.
\end{equation*}
Gluing of two objects give a gluing $X$ on $U$.

Replacing $U_1 \cup U_2$ by $(U_1 \cup \hdots \cup U_{n-1})$ in the proof of the uniqueness of gluings for a cover by $3$ open give uniqueness of $X$, up to isomorphism.

\end{proof}

\begin{proof}[proof of Theorem \ref{thm:inftystack}]
The functor $$\st(A)(-) : Open(\spc(A))^{op} \rightarrow \infty\dashmod Cat$$ is well-defined by Proposition \ref{pro:inftyverdier}. Moreover, it is a stack if and only if the natural map $F_{desc} : \st(A)(U) \rightarrow Desc(\{U_{\alpha} \})$ is an equivalence of $\infty\dashmod\otimes$-categories. 

As limits in $\otimes$-categories are computed at the level of underlying categories, the result is equivalent to Proposition \ref{pro:es} and Proposition \ref{pro:ff}.
\end{proof}

\part{Applications}

\section{Picard groups}

\subsection{Spectrum of units}

\begin{de}
Let $\gm$ be the functor
$$ \gm : stable\dashmod\infty\dashmod\otimes Cat \rightarrow \infty\dashmod grpd$$
which sends a stable-$\infty\dashmod\otimes$-category to the $\infty$-groupoid of invertible objects and invertible arrows between them.
\end{de}

\begin{pro}
The functor $\gm$ commutes with homotopy limits and finite homotopy colimits. Moreover, for any presentable $\C \in stable\dashmod\infty\dashmod\otimes Cat$, we have identifications
\begin{eqnarray*}
\pi_0(\gm(\C)) &=& Pic(\C) \\
\pi_1(\gm(\C)) &=& Aut_{\C}(\mathbb{1}) \\
\pi_i(\gm(\C)) &=& \pi_{i-1}(\hom_{\C}(\mathbb{1}, \mathbb{1}) \\
\end{eqnarray*}

Where $\mathbb{1}$ denotes the unit in $\C$, and $Pic$ the Picard group of $\C$, that is the group of invertible elements in $\C$.
\end{pro}

\begin{proof}
This was already observed in \cite[Section 2]{MS14}. The identifications of the homotopy groups of $\gm$ comes directly from the definition.
\end{proof}

\begin{thm} \label{thm:mayervietoris}
Let $A$ be a Margolis-Hopf algebra and $\{U_{\alpha} \rightarrow \spc(A)\}_{\alpha \in S}$ be a segmental covering.
There is a Mayer-Vietoris spectral sequence
$$ E_2^{s,t} \Rightarrow \pi_{t-s}(\gm(\stinf(A))$$
whose $E_2$-term is $$E_2^{s,t} = \bigoplus_{\alpha_0, \hdots \alpha_s \in S} \pi_{t}(\gm(\stinf(A)( \bigcap_{i=0}^s U_{\alpha_i})).$$
\end{thm}

\begin{proof}
Applying $\gm$ to the functor $$\gm : \C(U_{\alpha})^{op} \rightarrow \infty\dashmod grpd$$
where $\C(U_{\alpha})$ denotes the \v{C}ech nerve of the covering $\{U_{\alpha} \rightarrow U\}$ gives $\gm(\stinf(A))$ as the homotopy limit of 

\begin{align*}
\C(U_{\alpha}) & \rightarrow & \infty\dashmod grpd \\
V & \mapsto & \gm(\stinf(A)(V)) \\
\end{align*}

which can be computed via the bousfield-Kan spectral sequence of a cosimplicial simplicial set. This is precisely the desired spectral sequence.
\end{proof}

\subsection{Local detection}

\begin{pro} \label{pro:localdetection}
For all partition $1 = i_0 < i_1 < \hdots i_n = N$ of the Margolis homology in $A$, the natural map
$$\Pic(A) \rightarrow \bigoplus_{k=0}^{n-1} \Pic(U_{[i_k,i_{k+1}]})$$
is injective.
\end{pro}

\begin{proof}
The hypothesis on $U$ implies that the Verdier quotient $S/S_{U^c}$ has a calculus of fractions on both sides. Thus, the elements of $\pi_1(\gm(\stinf(A)(U)))$ are diagrams
$$ \xymatrix{ & X \ar[dl]_f \ar[dr]^g& \\ \mathbb{1} & & \mathbb{1} }$$
were $f$ and $g$ have cofiber supported on $U^c$.
Let $f : X \rightarrow \mathbb{1}$ be a map whose cofiber lies outside $U$. Suppose that $U$ correspond to some interval $[a,b]$. Then $f$ induces an isomorphism in $H(-;p_c)$, for all $a \leq c \leq b$. By dimensions, there is at most one such map when $X = \mathbb{1}$.
In general, if $f_1,f_2 : X \rightarrow \mathbb{1}$ are two isomorphisms over $U$, then $f_1 \circ f_2^{-1}$ is an isomorphism $\mathbb{1} \rightarrow \mathbb{1}$, so that $f_2$ is the inverse of $f_1$ and there is at most one map $X \rightarrow \mathbb{1}$ whose cofiber lies outside $U$. Thus, $\gm(U) = 0$. We conclude using that for all $V \subset U \subset \spc(S)$, $S(U)(V) = S(V)$, and repeatidly using the Mayer-Vietoris spectral sequence (long exact sequence here, since we are covering with two open) computing Picard groups.
\end{proof}

\begin{rk}
The work of Bob Bruner in \cite{Br14} gives a short exact sequence
$$ Pic(\stinf(\A(1))) \inj Pic(\stinf(\A(1))(U_{0}) \oplus Pic(\stinf(\A(1))(U_{1}) \surj \Z/4.$$
As we will see in the next section, this $\Z/4$ can be interpreted as a local Picard group as well ($Pic(\stinf(\A(1))(U_{0} \cap U_{1}))$). The fact that the sequence is right exact however seems to be particular to this case.
\end{rk}

\section{What happens for $\spc(\A(1))$?} \label{sec:toy}

We now turn to the study of the particular case of $\A(1)$. The objective of this section is to formulate some classical results about the stable category of modules over this specific Margolis-Hopf algebra in terms of tensor triangulated geometry, and to see how these classical results can be seen as consequences of our work.

\subsection{The spectrum of $\A(1)$}

Let's first compute the spectrum of $A=\A(1)$. Denote $S = \st(\A(1))$. It is well-known that the cohomology of $\A(1)$ can be given the following presentation

\begin{equation*}
H^{*,*}(\A(1)) = \frac{\F[v_0, \eta, \alpha, \beta]}{(v_0\eta, \eta^3, \eta \alpha, \alpha^2-v_0^2\beta)},
\end{equation*}
with $|v_0| = (1,1)$, $|\eta| = (1,2)$, $|\alpha| = (3,7)$, and $|\beta| = (4,12)$.

By Proposition \ref{pro:projadspec}, we have a homeomorphism

\begin{equation} \label{eqn:spcproj}
\spc(S) = \proj(H^{*,*}(\A(1))
\end{equation}

and this projective variety has three points:
$$S^{(1)} = (\eta, v_0),$$
$$S^{(0)} = (\eta,\alpha, \beta)$$
and 
$$S^{(01)} = (\eta).$$

\begin{pro}
Under the isomorphism \eqref{eqn:spcproj}, these correspond to the more familiar prime ideals (again denoted $S^*$):
\begin{itemize}
\item For $(v_0,\eta)$,
$$S^{(1)} \eq \st(\A(1))^{(1)}$$
the ideal consisting in $Q_1$-local modules.
\item For $(\eta,\alpha, \beta)$,
$$ S^{(0)} = \st(\A(1))^{(0)},$$
the ideal consisting of modules without Margolis $Q_1$-homology (or $Q_0$-local).
\item For $(\eta)$, the prime $S^{(01)}$ is the smallest thick $\otimes$-ideal containing the local modules. Thus it contains the finitely generated modules build from finitely generated $Q_0$-local modules and finitely generated $Q_1$-local modules. 
\end{itemize}
\end{pro}

\begin{proof}
Let $i=0,1$. The proper prime ideal $S^{(i)}$ contains $\Lambda(Q_0)$ or $\A(1)//\Lambda(Q_0)$ (if $i=1$ or $0$ respectively). In particular, it contains all $Q_i$-local modules (little computation by hand for $i=0$).
Let $X$ be a $i$-local module. Consider the commutative diagram
$$\xymatrix{ H^{*,*}(\A(1)) \ar[r] \ar[d]& H^{*,*}(\Lambda(Q_i)) \ar[d] \\
\ext_{\A(1)}^{*,*}(X,X) \ar[r] & \ext_{\Lambda(Q_i)}(X,X) }. $$
The vertical right arrow being zero concludes.\\

The third ideal contains both the $Q_0$ and $Q_1$-local modules, thus it contains every finitely generated module built from finitely generated local modules.

\end{proof}

\begin{rk}
The characterization of $S^{(01)}$ might seem quite strange to the reader who is already familiar with the stable category of $\A(1)$-modules. Indeed, without the smallness assumption, the smallest thick subcategory of $\st(\A(1))$ containing both the $Q_0$-local modules  and the $Q_1$-local modules is the entire category $\st(\A(1))$ (this is a by-product of Margolis' work).

However, this is not the case in our situation. Indeed, any finite dimensional $Q_i$-local $\A(1)$-module is even dimensional. Consequently, if $M$ is $S^{(01)}$, then it is build from modules that are even dimensional, and $M$ itself must be even dimensional.

In particular, the stable module $\mathbb{1} \not\in S^{(01)}$. This shows that the prime ideal is not $S^{(01)}$.
\end{rk}

The spectrum of $S$ is given in Figure \ref{fig:spca1}.

\begin{figure} \label{fig:spca1}
\definecolor{qqqqff}{rgb}{0.3333333333333333,0.3333333333333333,0.3333333333333333}
\begin{tikzpicture}[line cap=round,line join=round,>=triangle 45,x=1.0cm,y=1.0cm]
\clip(3.5482058641580387,0.7464304190172949) rectangle (8.815897067920975,5.034871334901224);
\draw (5.,2.)-- (6.,4.);
\draw (7.,2.)-- (6.,4.);
\draw (5.861262578630868,4.798500575758015) node[anchor=north west] {$S^{(01)}$};
\draw (4.780710536833342,1.911400589080252) node[anchor=north west] {$S^{(0)}$};
\draw (6.891163743469134,1.9282842147333383) node[anchor=north west] {$S^{(1)}$};
\begin{scriptsize}
\draw [fill=qqqqff] (5.,2.) circle (1.5pt);
\draw [fill=qqqqff] (6.,4.) circle (1.5pt);
\draw [fill=qqqqff] (7.,2.) circle (1.5pt);
\end{scriptsize}
\end{tikzpicture}
\caption{The spectrum of $\st(\A(1))$. Closure goes down along lines.}
\end{figure}
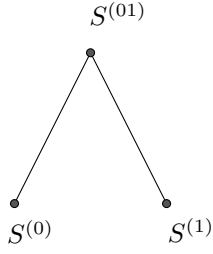

In particular, $S^{(01)}$ is a generic point of the spectrum.

\subsection{Localizations}

\begin{lemma}
There is a segmental cover of $\spc(S)$ by the open subsets $U_0 = \{ S^{(0)}, S^{(01)} \}$ and $U_1 = \{ S^{(1)}, S^{(01)} \}$.
The intersection $U_{01} = U_0 \cap U_1 = \{ S^{(01)} \}$.
\end{lemma}

We will now use the machinery of localizations and gluing to study the stable category of $\A(1)$-modules.

\begin{lemma}
The Verdier quotient $S(U_i)$ is isomorphic to the full subcategory of $S$ whose objects are $L_i( \mathbb{1}) \otimes X$, for some finitely generated $X$, where $L_i$ denotes the {\it $Q_i$-colocalization}.
\end{lemma}

\begin{proof}
This is a consequence of Lemma \ref{lemma:supported} and the explicit formula for localization given in \cite{Mar83}. Alternatively, here is a more down to earth approach. \\
We take the colocalization in both cases! Verdier localization at $S^{(i)}$ inverts the maps $f : X \rightarrow Y$ whose cone is $i$-local. In particular, one can choose a representative of any object to be the fiber of the localization at $1-i$. This defines a functor from the stable category to the local one which factors through the Verdier quotient. 

Now, this is an equivalence of categories by Margolis criterion. Note that we are allowed to use Margolis criterion in our situation because all the constructions considered in this proof stays in the category of bounded below-modules (respectively bounded-above modules) if $i = 0$ (respectively $i = 1$).
\end{proof}

\begin{pro}
The category $S(U_{01})$  is equivalent to the full subcategory of $S$ whose objects are $\mathbb{1}[x^{\pm 1}] \otimes X$ for a finitely generated $\A(1)$-module $X$.
\end{pro}

This rely on the following fact:

\begin{lemma}
The module $\mathbb{1}[x^{\pm 1}]$ is a field. Precisely, any finitely generated $\mathbb{1}[x^{\pm 1}]$-module in $\st(\A(1))$ is of the form $\mathbb{1}[x^{\pm 1}]\otimes X$.
\end{lemma}

\begin{proof}
The first observation is that any non trivial map $Hom_{\A(1)}(\mathbb{1}[x^{\pm 1}],\mathbb{1}[x^{\pm 1}])$ is the identity via the periodicity isomorphism $\Sigma^4\mathbb{1}[x^{\pm 1}]=\mathbb{1}[x^{\pm 1}]$. 
Let $X$ be a finite dimensional $\A(1)$-module. By $\Omega$-periodicity of $\mathbb{1}[x^{\pm 1}]$, the filtration of $X$ by degree gives a tower whose limit is $\mathbb{1}[x^{\pm 1}] \otimes X$ which splits.
\end{proof}

\subsection{Local detection}

Consider now the segmental cover
$$ \xymatrix{ U_0 \cap U_1 \ar[d] \ar[r] & U_0 \ar[d]\\
U_1 \ar[r] & U_0 \cap U_1. }$$

This gives rise to half a long exact sequence

\begin{eqnarray*}
\hdots \rightarrow \gm( S) \rightarrow \gm(S^{(1)}) \oplus \gm(S^{(0)}) \rightarrow \gm( S^{(01)}) \rightarrow Pic(S)  \\\rightarrow Pic(S^{(0)}) \oplus Pic(S^{(1)}) \rightarrow Pic(S^{(01)})
\end{eqnarray*}
$$ $$

But all $\gm$ are zero (these are automorphisms of the unit $\un$). Thus

\begin{pro}
There is an exact sequence
$$ Pic(S) \inj Pic(S^{(0)}) \oplus Pic(S^{(1)}) \rightarrow Pic(S^{(01)}).$$
\end{pro}

\bibliographystyle{alpha}
\bibliography{biblio}

\end{document}